\documentclass{amsart}
\usepackage[utf8]{inputenc}
\usepackage[english]{babel}
\usepackage{amsmath, amssymb, amsthm, amsfonts}
\usepackage[a4paper,margin=2cm,footskip=0.25in]{geometry}
\usepackage[usenames,dvipsnames,svgnames,table]{xcolor}
\usepackage{hyperref}
\usepackage{graphicx}
\usepackage[all]{xy}
\usepackage{tikz}
\usetikzlibrary{arrows, babel}
\usepackage{tikz-cd}
\usepackage{natbib}
\usepackage[capitalise,noabbrev]{cleveref}

\definecolor{BlueGreen}{HTML}{74a193}
\hypersetup{colorlinks=false, linkbordercolor=BlueGreen,linkcolor=BlueGreen}
\hypersetup{citecolor=BlueGreen}

\newtheorem{introtheorem}{Theorem}

\newtheorem{thm}{Theorem}[section]
\newtheorem{cor}[thm]{Corollary}
\newtheorem{lem}[thm]{Lemma}
\newtheorem{prop}[thm]{Proposition}
\theoremstyle{definition}
\newtheorem{defn}[thm]{Definition}
\theoremstyle{remark}
\newtheorem{rmk}[thm]{Remark}
\newtheorem{exm}[thm]{Example}

\tikzcdset{
cells={font=\everymath\expandafter{\the\everymath\displaystyle}},
}

\newcommand{\CA}{\mathcal{A}}

\newcommand{\CH}{\mathcal{H}}

\newcommand{\CL}{\mathcal{L}}
\newcommand{\CP}{\mathcal{P}}
\newcommand{\CQ}{\mathcal{Q}}

\newcommand{\CR}{\mathcal{R}}

\newcommand{\NR}{{\mathbf R}}
\newcommand{\NQ}{{\mathbf Q}}
\def\NN{\ensuremath{\mathbb{N}}}
\def\KK{\ensuremath{\mathbb{K}}}

\def\ZZ{\ensuremath{\mathbb{Z}}}
\def\QQ{\ensuremath{\mathbb{Q}}}
\newcommand{\Ab}{{\operatorname{Ab}}}

\newcommand{\join}{\vee}

\newcommand{\meet}{\wedge}
\newcommand{\Id}{\mathrm{id}}
\newcommand{\Fun}{\mathrm{Fun}}
\newcommand{\cyl}{\mathrm{cyl}}
\newcommand{\cocyl}{\mathrm{cocyl}}

\newcommand{\Ch}{\mathrm{Ch}}
\newcommand{\Mod}{{\operatorname{-mod}}}
\newcommand{\End}{\operatorname{End}}
\newcommand{\Aut}{\operatorname{Aut}}
\newcommand{\op}{{\mathrm{op}}}
\newcommand{\res}{\operatorname{Res}}

\newcommand{\direct}[1]{\overrightarrow{#1}}
\newcommand{\invers}[1]{\overleftarrow{#1}}
\renewcommand{\d}{\mathrm{d}}
\newcommand{\cd}{\mathrm{cd}}

\renewcommand{\ll}{\sqsubset}
\newcommand{\llneq}{\mathrel{\text{%
    \begin{tikzpicture}[baseline=-0.4ex, line width=0.55pt, line cap=round]
      \draw (0.65em, 1.3ex) -- (0, 1.3ex) -- (0, 0.1ex) -- (0.65em, 0.1ex);
      \draw (0, -0.5ex) -- (0.65em, -0.5ex);
      \draw (0.15em, -0.9ex) -- (0.5em, -0.1ex);
    \end{tikzpicture}%
  }}%
}

\DeclareMathOperator*\colim{colim}
\newcommand{\hocolim}{\operatorname{hocolim}}
\newcommand{\PSL}{\operatorname{PSL}}

\newcommand{\Bianchids}{1, 2}

\title{Cohomology of non-finite CL-shellable posets}
\newif\ifanon
\anonfalse 

\ifanon
  \author{Author Details Omitted for Double-Anonymous Review}
  \address{}
  \email{}
  \thanks{}
\else
\author[G. Carrión Santiago]{Guille {Carrión~Santiago}}
\address{Departamento de Matemática Aplicada, Ciencia e Ingeniería de los Materiales y Tecnología Electrónica, Universidad Rey Juan Carlos, 28933 Mótoles (Madrid), Spain}
\email{guille.carrions@urjc.es}
\author[A. Díaz ramos]{Antonio Díaz Ramos}
\address{Departamento de Ágebra, Geometría y Topología, Universidad de Málaga, Málaga, Spain}
\email{adiazramos@uma.es}
\fi

\date{\today}


\begin{document}
\begin{abstract}
Shellable complexes are homotopy equivalent to a wedge of spheres of possibly different dimensions, so that the (co)homology of the constant functor over the complex is concentrated in those degrees. In this work, we introduce the concept of a stable functor -a local weakening of fibrancy- over a shellable poset, which ensures the vanishing of the (co)homology of such a functor in specific degrees. The methods are based on a model category structure on the category of functors indexed by a filtered poset and the combinatorial structure of shellable posets. Our techniques work over non-finite and non-pure posets and employ a description of (co)homology via explicit fibrant replacements. Applications include acyclicity criteria for Mackey functors, computation of cohomology of $j$-th exterior powers over arrangement lattices, and homological decompositions for Bianchi groups $\Gamma_d$ for $d=1,2$.
\end{abstract}
\maketitle

\section*{Introduction}

Shellable simplicial and cellular complexes have proven to be a powerful framework for addressing problems in algebraic topology, commutative algebra, and combinatorics. For instance, they are closely related to Stanley–Reisner rings, since shellable complexes are Cohen–Macaulay \cite{Bjoerner1980}, or they may be employed to determine the homotopy type of Coxeter complexes and Tits buildings.
\medskip
 
The first definitions of shellability trace back at least to  \cite{MR90052} and \cite{MR97055}, and we refer the reader to  \cite{B1984b}, \cite{B95}, and \cite{MR1600042} for a historical account of shellability. A modern formulation of shellability is as follows: A finite pure $d$-dimensional simplicial complex $K$ is \emph{shellable} if its $d$-faces can be ordered $F_1,F_2,\dots,F_t$ so that, for all $1\le i<k\le t$, there exists $j<k$ with $F_i\cap F_k\subseteq F_j\cap F_k$ and $\dim(F_j\cap F_k)=d-1$, see \cite[p.1854]{B95}.
The finiteness assumption on the number of $d$-faces can be removed as shown in \cite{B1984b}.
In the 1980s, several alternative notions of shellability emerged for the finite case, such as EL-shellability \cite{Bjoerner1980,Bjoerner1982} and CL-shellability \cite{Bjoerner1983}. This latter notion is of a recursive nature, as Bruggesser and Mani's approach \cite{MR328944}, and it was later generalized by removing the purity condition \cite{BWnonpure}. In this work, we provide a recursive definition of dual CL-shellability for non-necessarily finite and non-necessarily pure posets, see \cref{def:recursive_coatom_ordering}. We work with dual CL-shellable posets instead of CL-shellable posets because the face lattice of a shellable complex is a dual CL-shellable poset, see \cite[Theorem 5.13]{BWnonpure}.\medskip

Recall that (co)homology of functors over a poset is the same thing as (co)homology of sheaves over the poset and, in turn, these two notions coincide with that of higher (co)limits of those functors. In this paper, we stick to this latter point of view. Thus, we study higher (co)limits over a dual CL-shellable poset $\CP$, and we employ certain model category structure, so that higher (co)limits may be computed via (co)fibrant replacements, see Theorems \ref{thm:description_model_category_Mop} and \ref{thm:description_model_category_M}. With this setup, a functor $F\colon \CP^\op\to R\Mod$ is fibrant if and only if the natural morphism
\begin{equation*}
F(p)\to \lim_{\CP_{<p}} F
\end{equation*}
is surjective for every $p\in \CP$. The main contribution of this work is to have been able to isolate a condition that implies the vanishing of the higher limits at a single degree. More precisely, we say that a functor $F\colon \CP^\op\to R\Mod$ is \emph{stable} at degree $i$ if, for every chain $c=c_0\prec c_1\prec \dots \prec c_i=\hat 1$, the natural morphism 
\[
F(c_0)\to \lim_{\langle C(c) \rangle} F
\]
is surjective, where $\hat 1$ is the maximum of $\CP$ and  the sub-poset $\langle C(c)\rangle\subseteq \CP_{<c_0}$ is part of the CL-shellable structure of the poset $\CP$, see Definitions \ref{def:stability} and \ref{def:costability}. Below we state the main result of this paper, where $\d(\CP)$ stands for the degree or dimension of the poset $\CP$, see \cref{sec:Background}.\medskip

\begin{introtheorem}
\label{intro:main_theorem}
Let $\CP$ be a dual CL-shellable poset, $1\le i\le \d(\CP)-1$, and $F\colon \CP^{\op}\to R\Mod$ a stable functor at degree $i$. Then
\[
H^i(\CP\setminus \{\hat 1\};F)=0.
\]
\end{introtheorem}

 This is proven later as \cref{main_theorem} and its version for higher colimits as \cref{dual/main_theorem}. If $F$ is not stable at degree $i=\d(\CP)-2$, we still provide mild conditions to determine the top cohomology group $H^i(\CP\setminus \{\hat 1\};F)$, see Theorems \ref{thm:exact_sequence_for_atoms} and \ref{thm:exact_sequence_for_atoms_covariant_colimit}. We exploit these results in the following applications.\medskip

\ifanon
Weak
\else
In authors' previous work \cite{CD23}, weak
\fi 
 Mackey functors over posets are introduced by mimicking the classical notion.
\ifanon
by Carrión and Díaz \cite{CD23}.
\fi In this paper, we show how the machinery of dual CL-shellable posets together with the stability condition allow us to improve  \cite[Theorem B]{CD23}. This latter result implies that all higher limits of a weak Mackey functor vanish if there exists a so-called ``global'' quasi-unit. On the contrary, below we assume only the existence of ``local'' quasi-units which, in turn, implies the stability condition, and hence the vanishing of a single higher limit.

\begin{introtheorem}
\label{intro:Mackey}
Let $\CP$ be a dual CL-shellable poset, and $G\colon \CP^\op\to R\Mod$ a weak Mackey functor such that, for every chain $c=c_0\prec c_1\prec \dots \prec c_i=\hat 1$, $G$ has a quasi-unit in $\{c_0\}\cup \langle C(c)\rangle$. Then
\[
H^i(\CP\setminus \{\hat 1\};G)=0.
\]
\end{introtheorem}

This is proven as \cref{thm:Mackey}. The second application presented here is related to the homology of hyperplane arrangements. Earlier vanishing results for local sheaves on arrangement lattices were obtained by Yuzvinsky in \cite{MR1062840}. Here, we revisit the computation by Everitt and Turner \cite{Everitt2019,Everitt2022} concerning the homology of the exterior power functor over the arrangement lattice associated to a set of hyperplanes over a finite-dimensional vector space. Below, we reproduce some of those results under the milder assumption that the number of hyperplanes is numerable and not necessarily finite. The proofs in those two papers are based on the deletion-restriction long exact sequence in homology and the Boolean cover of the lattice. Here, we employ dual CL-shellability for posets, stable functors, and basic linear algebra.

\begin{introtheorem}
\label{intro:hyperplane arrangement}
Let $V$ be a finite-dimensional vector space, $H$ be a numerable set of hyperplanes of $V$, and $\CL(H)$ the arrangement lattice of $H$. Then $d_{i,j}=\dim H_i(\CL(H)\setminus\{V\};\Lambda^j(-))$ is zero outside of the set $A\cup B$, where $A$ is the segment
\[
A=\{(0,j)|0\leq j\leq \dim(V)-1\}
\]
and $B$ is the truncated band
\[
B=\{(i,j)\text{ such that } 1\leq i\leq \d(\CL(H))-2 \text{ and }\d(\CL(H))-1\leq i+j\leq \dim(V)-1\}.
\]
Moreover, $d_{0,j}=\begin{pmatrix} \dim(V)\\j \end{pmatrix}$ if $0\leq j\leq \d(\CL(H))-2$ and $d_{0,\dim(V)-1}=\#H$. In addition, if $H$ is finite, $\dim(\hat 0)=0$, $1\leq i\leq \d(\CL(H))-2$, and $i+j=\d(\CL(H))-1$, then $d_{i,j}$ equals
\[
(-1)^{i+1}\begin{pmatrix} \dim(V)\\j \end{pmatrix} + \sum_{d=\d(\CL(H))-1-i}^{\d(\CL(H))-1} (-1)^{d-\d(\CL(H))+1+i} \begin{pmatrix} d\\j \end{pmatrix} \sum_{\substack{x\in \CL(H)\\\d(x)=d}} |\mu(x,V)|.
\]
\end{introtheorem}

Finally, as third application, we provide an integral homology decomposition of the classifying space of the Bianchi group $\Gamma_d=\PSL_2(\mathcal O_d)$ for $d=\Bianchids$, where $\mathcal O_d$ is the ring of integers of the quadratic field $\QQ(\sqrt{-d})$. The following theorem generalizes \cite[Theorem D]{CD23}.

\begin{introtheorem}\label{thm:integral_homology_decomposition_Bianchi}
For $d=\Bianchids$, there exists a poset $\CP_d$ of proper subgroups of $\Gamma_d$ and a homotopy equivalence
\[
\hocolim_{U\in \CP_d} B(U)\stackrel{\simeq}\longrightarrow B(\Gamma_d).
\]
\end{introtheorem}

Here, the morphisms in $\CP_d$ consist of the inclusions between the given subgroups. To prove this result, we show first that the homology of the homotopy fiber of the map above is given by the homology of certain functor $H$. Then we equip $\CP_d$ with a structure of dual CL-shellable poset and show that $H$ is stable with respect to this structure at every possible degree. As a direct consequence of this theorem, we can determine the homology
and cohomology of $\Gamma_d$ using the classical Bousfield–Kan spectral sequence, see \cite[Example 4.2]{CD23} for the case of $\Gamma_1$. These cohomology groups have been determined elsewhere via other methods, see \cite{BerkoveIntegral}.\bigskip

\textbf{Outline of the paper:} Section \ref{sec:Background} contains background on bounded and filtered posets, as well as the definition of dual CL-shellable posets and some of their categorical properties. In Section \ref{section:(co)homology}, we consider certain model categories structures in the categories of functors indexed by a filtered poset. Then we consider explicit (co)fibrant replacements and show that they determine (co)homology, see Propositions \ref{prop:hig-lim-via-RF} and \ref{prop:hig-colim-via-QF}. The (co)stability property at a given degree is introduced in Section \ref{sec:vanishing} and, in this same section,   \cref{intro:main_theorem} and its dual are proven employing the earlier (co)fibrant replacements. The top (co)homology groups of a functor over a dual CL-shellable poset are studied in Section \ref{sec:non-vanishing_stable_functor}. In the final section, the aforementioned applications are developed.
\bigskip

\ifanon
\noindent{\bf Acknowledgments.} The acknowledgments are omitted for double-anonymous review
\else
\noindent{\bf Acknowledgments.} We would like to thank Prof. Russ Woodroofe for his comments and historical clarifications, which helped us improve the introduction of the paper. The first author is supported by Universidad de Málaga grant G RYC-2010-05663, Comissionat per Universitats i Recerca de la Generalitat de Catalunya (grant No. 2021-SGR-01015), and  MICINN grants BES-2017-079851, PID2020-116481GB-I00, and PID2023-149804NB-I00. Both authors are supported by MICINN grant PID2020-118753GB-I00. The second author is supported by Junta de Andalucía grant ProyExcel00827.
\fi

\section{Background on CL-shellable posets}
\label{sec:Background}
Given a poset $\CP=(\CP,\le)$, we denote by  $\prec$ the \emph{cover relation} in $\CP$, i.e., $p\prec q$ if and only if $p< q$ and
\[
\forall r\in \CP, p\le r\le q\Rightarrow p=r\text{ or }q=r.
\]
A chain $c$ of \emph{length} $n$ in $\CP$ is a sequence of totally ordered elements,
\[ 
c_0<c_1<\dots<c_n.
\]
We say that $c$ is a \emph{$c_n$-chain} to emphasise its largest element, and such a chain is \emph{unrefinable} if $c_i\prec c_{i+1}$ for any $0\leq i\leq n-1$. The poset $\CP$ is \emph{filtered} if the \emph{degree} map $\d\colon \CP\to \NN$ given by 
\[
\d(p)=\max\{n\mid c_0\prec c_1\prec \dots \prec c_n=p\}
\]
is strictly order preserving, and we define the \emph{degree} of a filtered poset $\CP$ as the maximum of its degrees
\[
\d(\CP)=\max\{\d(p)\mid p\in \CP\}\in \NN\cup \{\infty\}.
\] 
A filtered poset is said to be \emph{bounded} if it has a maximum $\hat 1$ and a minimum $\hat 0$. In that case, we denote by  $\overline \CP$ the subposet $\CP\setminus\{\hat0,\hat 1\}$, and it is immediate that
\[
\d(p)=\max\{n\mid \hat 0= c_0\prec c_1\prec \dots \prec c_n=p\},
\]
and that $\d(\CP)=\d(\hat1)$. We can also define a \emph{co-degree} function $\cd\colon \CP\to \NN$ defined by
\[
\cd(p)=\min\{n\mid  p=c_0\prec c_1\prec\dots \prec c_n=\hat 1\}.
\]
A poset is said to be \emph{pure} if all maximal chains have the same length. In that case, it satisfies the Jordan-Dedekind chain condition \cite[p.11]{Birkhoff}, $\cd$ is order-reversing, and, for all $p\in \CP$,
\[
\d(p)+\cd(p)=\d(\CP).
\]
These conclusions do not hold in general for a non-pure poset.

\begin{rmk}
Standard literature often refers to a finite, pure, and bounded poset $\CP$ as graded, and, in that case, $\CP$ is filtered in the sense defined above and the degree map $d$ is termed rank function, see for instance  \cite{Bjoerner1983}.
\end{rmk}

An \emph{atom} $a$ of a bounded poset $\CP$ is an element that covers the minimum $\hat 0 \prec a$. Dually, a \emph{coatom} $h$ is an element covered by the maximum $h\prec \hat 1$. Thus, atoms (resp. coatoms)  are the objects of degree (resp. co-degree) $1$. For every $p\in \CP$, we denote by $\CP_{\le p}$ the sub-poset of $\CP$ consisting of those $q\in \CP$ such that $q\le p$, 
\[
\CP_{\le p}=\{q\in \CP\mid q\le p\}.
\]
Similarly, we define $\CP_{\ge p}$, $\CP_{> p}$, $\CP_{< p}$ and $\CP_{\prec p}$.
If $Q$ is a subset of $\CP$, we denote by $\langle Q\rangle_\CP $ the sub-poset of $\CP$ whose elements are those that are less than some element in $Q$,
\[
\langle Q\rangle_\CP=\{p\in \CP\mid p\le q\text{ for some }q\in Q\},
\]
and we write $\langle Q\rangle$ if the poset $\CP$ is understood. Next, we generalise the notion of dual CL-shellable posets 
for bounded posets (not necessarily finite) via the characterisation given by Björner and Wachs \cite[Definition 5.10, Theorem 5.11]{BWnonpure}. So, given a bounded poset $\CP$, a \emph{coatom ordering} is a family $\{\ll_c\}_c$ of well-orders $\ll_c$ of $\CP_{\prec c_0}$,
where $c$ ranges over the family of all unrefinable $\hat1$-chains. In addition, for such a chain $c=c_0\prec c_1\prec \ldots\prec \hat 1$, and $c'=c_1\prec \ldots\prec \hat 1$, we define the set
\[
C(c)=\{x\prec c_0\mid x< h\prec c_1\mbox{ for some } h\llneq_{c'} c_0\}
\]
if the length of $c$ is positive, and $C(\hat 1)=\emptyset$ otherwise. Note that $C(c)$ is also empty if either $c_0=\hat 0$ or $c_0$ equals the least element in $\ll_{c'}$. For a coatom $h$, we write ${C(h)=C(h\prec \hat 1)}$ for simplicity,  and record that
\begin{equation}\label{equ:C(h)_h_cotaom}
C(h)=\{x\prec h\mid x< h'\prec \hat{1}\mbox{ for some } h'\llneq_{\hat 1} h\}.
\end{equation}
\begin{defn}
\label{def:recursive_coatom_ordering}
For $\CP$ be a bounded poset, a \emph{recursive coatom ordering} is a coatom ordering  $\{\ll_c\}_c$ such that, for every unrefinable $\hat1$-chain $c$,
\begin{description}
\item[CL0] Every initial segment of the linear order $\ll_c$ is either finite or isomorphic to $(\NN,\le)$. 
\item[CL1] The set $C(c)$ is an initial segment of the linear order $\ll_c$.
\item[CL2] For $h'\ll_c h\in \CP_{\prec c_0}$ and $p\in \CP_{<h}\cap \CP_{< h'}$, there exist $h''\in \CP_{\prec c_0}$ and $q\in \CP_{\prec h}$ with $h''\ll_c h$ and $p \le q< h''$.
\end{description}
In that case, we say that the pair $(\CP,\ll)$ is \emph{dual CL-shellable}.
\end{defn}
We say that $\CP$ is \emph{dual CL-shellable} if $(\CP,\ll)$ is \emph{dual CL-shellable}  for some $\ll$. If $\CP$ is finite, then \textbf{CL0} trivially holds, and our notion of dual CL-shellable poset coincides with that in \cite{BWnonpure}. Now, we present some properties of CL-shellable posets that we will need later.

\begin{lem}
\label{lema:Q_final_category}
If $\CP$ is a dual CL-shellable poset, then the subposet of objects of codegrees $1$ and $2$ is final in $\CP$.
\end{lem}
\begin{proof}
Let $\{\ll_c\}_c$ be a recursive coatom ordering for $\CP$, $\CQ$ the subposet of objects of codegrees $1$ and $2$, and $\CH=\CP_{\prec \hat 1}\subseteq \CQ$ the set of coatoms of $\CP$.  For simplicity we denote by $\ll$ the linear order $\ll_{\hat1}$. We need to prove that for every $p\in \CP$, the set $\CQ\cap\CP_{\geq p}$ is connected. Thus, let $h_0,h_1\in \CQ\cap \CP_{\geq p}$ and, without loss of generality, assume that $h_0, h_1\in \CH\cap\CP_{\geq p}$. We show that $h_0$ and $h_1$ lie in the same connected component in $\CQ \cap \CP_{\geq p}$.

Set $H=\{h_0,h_1\}$. If $h_0=h_1$ we are done. Otherwise, we note that $H\subseteq \CH\cap \CP_{\geq p}$ and we repeat the following process: apply \textbf{CL2} to any pair of distinct elements $h,h'\in H$ with $h'\ll h$, and add the resulting element $h''\in \CH\cap \CP_{\geq p}$ to this set. As the set $\{h\in \CH\mid h\ll h_0\text{ or }h_1\}$ is finite by \textbf{CL0}, we obtain a set $H'\subseteq \CH\cap \CP_{\geq p}$ which contains $H$ and is closed under this process. Now, by construction, it is easy to see inductively that the subset of the $n$ smallest elements in $H'$ is connected in $\CQ\cap \CP_{\ge p}$ for $n=1,2,3,\ldots,|H'|$.
\end{proof}

\begin{defn}
\label{def:Q_compatible_recursive_ordering}
Let $(\CP,\ll)$ be a dual CL-shellable pair, $p\in \CP$, and $Q\subset P_{\prec p}$. For an unrefinable $\hat 1$-chain $c$ with $c_0=p$, we say that $Q$ is \emph{$c$-compatible} if $Q$ is an initial segment in $\ll_{c}$. We say that $Q$ is \emph{compatible with the recursive coatom ordering} if $Q$ is $c$-compatible for some $c$.
\end{defn}
\begin{lem}
\label{lem:Q_CL-shellable}
Let $\CP$ be a dual CL-shellable poset, $p\in \CP$, and $Q\subset \CP_{\prec p}$ compatible with the recursive coatom ordering. Then, $\langle Q\rangle\cup\{p\}$ is a dual CL-shellable poset of degree $\d(p)$.
\end{lem}
\begin{proof}
Let $c$ be an unrefinable chain from $p$ to $\hat 1$ such that $Q$ is an initial segment in $\ll_c$ and write $\CP'=\langle Q\rangle\cup\{p\}$. If $c'$ is an unrefinable $p$-chain in $\CP'$, we define the linear order $\ll'_{c'}$ as the restriction to $\CP'_{\prec c'_0}\subset \CP_{\prec c'_0}$ of the linear order $\ll_{c'\prec c}$, where 
\[
c'\prec c=c'_0\prec \dots \prec p\prec c_1\prec \dots \prec \hat 1.
\]
Then $\{\ll'_{c'}\}_{c'}$ is a recursive coatom ordering for $\CP'$: If $c'\neq p$, then $\CP'_{\prec c'_0}= \CP_{\prec c'_0}$ and the axioms hold as they are valid for $\{\ll_c\}_c$. If $c'=p$, then \textbf{CL0} and \textbf{CL1} trivially hold and \textbf{CL2} holds because $Q$ is an initial segment in $\ll_p$.
\end{proof}

\section{(Co)homology of functors}\label{section:(co)homology}
\ifanon
As shown in Carrión-Diaz 
\else
As shown in the authors' previous 
\fi
work \cite{CD23}, higher (co)limits can be described in terms of (co)fibrant objects. In this section, we describe the cohomology (resp. homology) of a filtered poset $\CP$ with coefficients in a covariant (resp. contravariant) functor $F\colon \CP \to R\text{-}\mathrm{Mod}$ in terms of cofibrant (resp. fibrant) replacements, which we construct explicitly.

\subsection{Cohomology}
Let $\CP$ be a filtered poset and $R$ be a commutative ring with unit. For functors in $\Fun(\CP^\op,R\Mod)$, taking limit defines a functor,
$
\lim\colon \Fun(\CP^{op},R\Mod)\to R\Mod
$,
which is left exact but is not right exact, and thus it is meaningful to consider its right-derived functors. Higher limits of a functor $F$ are then the right-derived functors of $\lim$ evaluated in $F$, and we will denote them by
\[
H^*(F;\CP):=R^*\lim(-)(F). 
\]
The following statement is a generalization of the dual version of \cite[Theorem 2.1]{CD23} in which higher limits can be computed through fibrant replacements.

\begin{thm}
\label{thm:description_model_category_Mop}
Let $R$ be a commutative ring with unit and $\CP$ be a filtered poset. There exists a model category structure in the category of functors $\Fun(\CP^\op,\Ch(R))$ in which
\begin{enumerate}
\item a natural transformation $\eta\colon F\to G$ is a weak equivalence if, for every $p\in \CP$, the morphism  $\eta_p\colon F(p)\to G(p)$ is a quasi-isomorphism, that is $H^*(\eta_p)$ is an isomorphism;
\item a functor $F\colon \CP^\op\to \Ch(R)$ is fibrant if the natural map $F(p)\to \lim_{\CP_{<p}} F$ is a degree-wise epimorphism; and 
\item a functor $F\colon\CP^\op\to R\Mod$ satisfies 
\[
H^*(\CP; F)=H^*({\lim}_\CP \NR F),
\]
where ${\NR F\colon \CP^\op\to \Ch(R)}$ is any fibrant replacement of $F$.
\end{enumerate}
\end{thm}
\begin{proof}
We see $\CP$ as Reedy category, see \cite[Section 15]{Hirschhorn2003} for background, by setting $\direct\CP=\CP$ and $\invers{\CP}$ equal to the discrete category $\CP^0$, and we endow $\Ch(R)$ with the projective model category structure, see \cite[Subsection 18.4]{MAYPONTO2011}. Hence, the weak equivalences and the fibrations in $\Ch(R)$ are, respectively, the quasi-isomorphisms and the epimorphisms. Therefore, by the result of D. M. Kan \cite[Theorem 15.3.4]{Hirschhorn2003}, we can equip $\Fun(\CP^\op,\Ch(R))$ with a model category structure satisfying points (1) and (2) of the statement and the following,
\begin{center}
$\eta\colon F\to G$ is a cofibration if and only if  $\eta_p$ is a cofibration in $\Ch(R)$ for every $p\in \CP$.
\end{center}
Now, the adjoint pair, $\Delta\colon \Ch(R)\Longleftrightarrow
\Fun(\CP^\op,\Ch(R))\colon \lim$, where $\Delta$ is the diagonal functor, is a Quillen pair as, by the comments above, $\Delta$ preserves acyclic cofibrations and cofibrations. Thus, by \cite[Lemma 8.5.9]{Hirschhorn2003}, 
the total right derived functor ${\operatorname{Ho}(\Fun(\CP,\Ch(R)))\to \operatorname{Ho}(\Ch(R))}$ exists and takes $[F]$ to $[\lim \NR F]$, where $\NR F$ is a fibrant replacement of $F$. As there exist enough injectives in $\Fun(\CP^\op,\Ch(R))$, we recover its usual $i$-th right derived functors as ${H^i(\CP, F)=H^i(\lim \NR F)}$. This proves (3) of the statement.
\end{proof}

\begin{defn}
\label{def:mapping-cocylinder}
Let $f\colon C\to D$ be a map between not necessarily bounded cochain complexes. The mapping cocylinder of $f$ is the cochain complex $\cocyl(f)$ whose degree $n$ part is
\[
\cocyl(f)^n:=C^n\times D^{n-1}\times D^n,
\]
and whose differential is given by the formula
\[
\partial(c,d,d')=(\partial c,d'-f(c)-\partial d,\partial d').
\]
\end{defn}
By analogy with the mapping cylinder, see \cite[Section 1.5]{Weibel1994}, we have the following result.
\begin{prop}
\label{prop:cocyl-factorisation}
Let $f\colon C\to D$ be a morphism between cochain complexes. Then, there is a factorisation of $f$:
\[
\begin{tikzcd}[row sep=tiny]
C\arrow[r,"i"]&\cocyl(f)\arrow[r,"\pi"]& D\\
c\arrow[r,maps to] &(c,0,f(c))&\\
&(c,d,d')\arrow[r,maps to] &d'
\end{tikzcd}
\]
where $i$ is a monomorphism inducing an isomorphism in cohomology, and $\pi$ is a split epimorphism.
\end{prop}
\begin{defn}
\label{def:fibrant-replacement}
Let $F\colon \CP^\op\to R\Mod$ be a functor over a filtered poset. The \emph{cocylinder} of $F$ is the functor $\NR{F}\colon \CP^\op\to \Ch(R)$ defined inductively on the objects by
\[
\NR F(p)=\cocyl\left(F(p)\to \lim_{\CP_{<p}} F\to \lim_{\CP_{<p}}\NR F\right)
\]
and for any $q<p$, the morphism
\[
\NR F(q<p):\NR F(p)\overset\pi\longrightarrow \lim_{\CP_{<p}}\NR F\longrightarrow \NR F(q)
\]
is the composite of the projection of the mapping cocylinder $\pi$ followed by the natural map of the cone $\lim_{\CP_{<p}}\NR F\to \NR F(q)$.
\end{defn}
\begin{prop}
\label{prop:hig-lim-via-RF}
Let $\CP$ be a filtered poset and $F\colon \CP^{\op}\to R\Mod$ be a functor. Then, the cocylinder $\NR{F}$ of $F$ is a fibrant replacement of $F$. Moreover, for every $p\in \CP$
\[
H^*(\CP_{<p};F)=H^*({\lim}_{\CP_{<p}}\NR{F})
\]
\end{prop}
\begin{proof}
By \cref{prop:cocyl-factorisation}, for every $p\in \CP$ we have the following factorisation of the composite $\varepsilon_p\colon F(p)\to \lim_{\CP_{<p}} F\to \lim_{\CP_{<p}} \NR F$
\[
F(p)\to \NR F(p):=\cocyl(\varepsilon_p)\to \lim_{\CP_{<p}} \NR F
\]
as a quasi-isomorphism followed by a split epimorphism. From \cref{thm:description_model_category_Mop}(1), the natural transformation $F\to \NR F$ is a weak equivalence and, from (2), $\NR F$ is a fibrant object. Moreover, since this construction only depends on the rays $\CP_{<p}$, for every $p\in \CP$, it follows that 
\[
(\NR F)\vert_{\CP_{< p}}=\NR (F\vert_{\CP_{< p}})
\]
for every $p$. Now, by statement (3) of the same theorem, we have $H^*(\CP_{<p};F)=H^*(\lim_{\CP_{<p}}\NR{F})$.
\end{proof}
\begin{prop}
\label{fib-rep-F(0)=0}
Let $\CP$ be a bounded poset, $p\in \CP$ and $F\colon \CP^{\op}\to R\Mod$ be a functor. Then, $(\NR F)\vert_{\overline \CP}\cong \NR (F\vert_{\overline \CP})$ if and only if $F(\hat 0)=0$. Moreover, in this case, $H^i(\CP\setminus\{\hat 1\};F)\cong H^i(\overline{\CP};F)$ for all $i\geq 0$.
\end{prop}
\begin{proof}
This follows from the inductive construction of $\NR F$ and the fact that $\NR F(p)=F(p)$ for every $p\in \CP$ of degree $1$ if and only if $F(\hat 0)=0$.
\end{proof}
\subsection{Homology}
In this subsection, we present the dual versions of the definitions and results introduced previously. Proofs are omitted. Taking colimit defines a right exact functor that is not left exact,
$\colim\colon \Fun(\CP,R\Mod)\to R\Mod$, so we can left-derive it. Higher colimits of a functor, or the homology modules of a functor, $F\colon \CP\to R\Mod$ are the left-derived functor of $\colim$ evaluated in $F$, and we will denote it by
\[
H_*(\CP;F)=L_*\colim(-)(F).
\]

\begin{thm}
\label{thm:description_model_category_M}
Let $R$ be a commutative ring with unit and $\CP$ be a filtered poset. There exists a model category structure in the category of functors $\Fun(\CP,\Ch(R))$ in which
\begin{enumerate}
\item a natural transformation $\eta\colon F\to G$ is a weak equivalence if, for every $p\in \CP$, the morphism  $\eta_p\colon F(p)\to G(p)$ is a quasi-isomorphism, that is $H_*(\eta_p)$ is an isomorphism;
\item a functor $F\colon \CP\to \Ch(R)$ is cofibrant if the natural map $\colim_{\CP_{<p}} F\to F(p)$ is a degree-wise monomorphism; and 
\item a functor $F\colon\CP\to R\Mod$ satisfies 
\[
H_*(\CP; F)=H_*({\colim}_\CP \NQ F),
\]
where ${\NQ F\colon \CP\to \Ch(R)}$ is any cofibrant replacement of $F$.
\end{enumerate}
\end{thm}

\begin{defn}
Let $f\colon C\to D$ be a map between not necessarily bounded chain complexes. The mapping cylinder of $f$ is the chain complex $\cocyl(f)$ whose degree $n$ part is
\[
\cyl(f)_n:=C_n\times C_{n-1}\times D_n,
\]
and whose differential is given by the formula
\[
\partial(c,c',d)=(\partial c+c',-\partial c',\partial d-f(b')).
\]
\end{defn}

\begin{prop}
\label{dual:cyl-factorisation}
Let $f\colon C\to D$ be a morphism between chain complexes. Then, there is a factorisation of $f$:
\[
\begin{tikzcd}[row sep=tiny]
C\arrow[r,"i"]&\cyl(f)\arrow[r,"\pi"]& D\\
c\arrow[r,maps to] &(c,0,0)&\\
&(c,c',d)\arrow[r,maps to] &f(c)+d
\end{tikzcd}
\]
where $i$ is a monomorphism inducing an isomorphism in homology, and $\pi$ is a split epimorphism.
\end{prop}
\begin{defn}
\label{def:cofibrant-replacement}
Let $F\colon \CP\to R\Mod$ be a functor over a filtered poset. The \emph{cylinder} of $F$, is the functor $\NQ{F}\colon \CP\to \Ch(R)$ defined inductively on the objects by
\[
\NQ F(p)=\cyl\left(\colim_{\CP_{<p}}\NQ F\to \colim_{\CP_{<p}} F\to F(p)\right)
\]
and for any $q<p$, the morphism
\[
\NQ F(q<p):\NQ F(q)\longrightarrow \colim_{\CP_{<p}}\NQ F\overset i\longrightarrow \NQ F(p)
\]
is the composite of the natural map of the cocone $F(q)\to \colim_{\CP_{<p}}\NQ F$ followed by the inclusion of the mapping cylinder $i$.
\end{defn}
\begin{prop}
\label{prop:hig-colim-via-QF}
Let $\CP$ be a filtered poset and $F\colon \CP\to R\Mod$ be a functor. Then, the cylinder $\NQ{F}$ of $F$ is a cofibrant replacement of $F$. Moreover, for every $p\in \CP$
\[
H_*(\CP_{<p};F)=H_*({\colim}_{\CP_{<p}}\NQ{F})
\]
\end{prop}
\begin{prop}
Let $\CP$ be a bounded poset, $p\in \CP$ and $F\colon \CP\to R\Mod$ be a functor. Then, $(\NQ F)\vert_{\overline \CP}\cong \NQ (F\vert_{\overline \CP})$ if and only if $F(\hat 0)=0$.  Moreover, in this case, $H_i(\CP\setminus\{\hat 1\};F)\cong H_i(\overline{\CP};F)$ for all $i\geq 0$.
\end{prop}

\section{Vanishing (co)homology of stable functors}\label{sec:vanishing}
In this section, we present a notion of stable functor, which is weaker than the notion of cofibrant functor presented in Section \ref{section:(co)homology}, and that still implies vanishing of the cohomology in a single degree. This notion is inspired by the recursive construction of shellable posets, see \cite{BWnonpure}. We also introduce the dual notion of co-stable functor and the corresponding vanishing results for homology.

\begin{defn}
\label{def:stability}
Let $(\CP,\ll)$ be a dual CL-shellable pair and $F\colon \CP^{\op}\to R\Mod$ a functor. We say that $(\CP, \ll,F)$ has the \emph{stability property at} $1\le i\leq \d(\CP)-1$ if for any unrefinable $\hat1$-chain $c$ of length $i$, the natural map
\[
F(c_0)\to \lim_{\langle C(c) \rangle} F
\]
is an epimorphism. If $(\CP,\ll)$ is clear from the context, we also say that $F$ is \emph{stable} at $i$.
\end{defn}

Note that, if $\CP$ is pure, then the previous condition is equivalent to that the natural map above is surjective for any unrefinable $\hat1$-chain $c$ with $\cd(c_0)=i$.

\begin{rmk}
\label{rmk:stability_property_cases}
The condition above for either $i=0$ or $i=\d(\CP)$ holds trivially as $C(c)=\emptyset$ in these cases. If $i=\d(\CP)-1$, then the natural map above is either $F(c_0) \to F(\hat 0)$ or $F(c_0)\to 0$. If $F(\hat 0)=0$ and $i=\d(\CP)-2$, the natural map above is exactly
\[
F(c_0)\to \bigoplus_{a\in C(c)} F(a),
\]
where we note that $C(c)$ is contained in the set of atoms of $\CP$. 
\end{rmk}

\begin{thm}
\label{main_theorem}
Let $(\CP,\ll)$ be a dual CL-shellable pair, $1\le i\le \d(\CP)-1$, and $F\colon \CP^{\op}\to R\Mod$. If $(\CP,\ll,F)$ has the stability property at $i$, then
\[
H^i(\CP\setminus \{\hat 1\};F)=0.
\]
\end{thm}

For covariant functors, we have the dual counterpart of \cref{def:stability} and \cref{main_theorem}.

\begin{defn}
\label{def:costability}
Let $(\CP,\ll)$ be a dual CL-shellable pair and $F\colon \CP\to R\Mod$ a functor. We say that $(\CP, \ll,F)$ has the \emph{co-stability property at} $1\leq i\leq \d(\CP)-1$ if for any unrefinable $\hat1$-chain $c$ of length $i$, the natural map
\[
\colim_{\langle C(c) \rangle} F\to F(c_0)
\]
is a monomorphism. If $(\CP,\ll)$ is clear from the context, we also say that $F$ is \emph{co-stable} at $i$.
\end{defn}

\begin{thm}
\label{dual/main_theorem}
Let $(\CP,\ll)$ be a dual CL-shellable pair, $1\le i\le \d(\CP)-1$, and $F\colon \CP\to R\Mod$. If $(\CP,\ll,F)$ has the co-stability property at $i$, then
\[
H_{i}(\CP\setminus \{\hat 1\};F)=0.
\]
\end{thm}

The proofs of \cref{main_theorem} and \ref{dual/main_theorem} are dual, and thus below we only spell the details for the former one. We start with some auxiliary lemmas and notation: Let $\CP$ be a poset, and $\CQ$ be a subposet of $\CP$. Given $F:\CP^\op\to R\Mod$, the \emph{restriction morphism} induced by the inclusion $\CQ\rightarrow \CP$ is denoted by
\[
\res^{\CP}_{\CQ}\colon \lim_\CP  F\to \lim_{\CQ}F.
\]

\begin{lem}
\label{auxiliar_seccion}
Let $(\CP,\ll)$ be a dual CL-shellable pair, $Q\subset \CP_{\prec \hat 1}$ be an initial segment of $\ll_{\hat1}$ and $h$ be the least coatom in $\CP_{\prec \hat 1}\setminus Q$. Let $F\colon \CP^{\op}\to R\Mod$ be a functor such that the morphism
\[
F(h)\to \lim_{\langle C(h)\rangle}F
\]
is a split epimorphism. Then the restriction
\[
\lim_{\langle Q\cup \{h\}\rangle}F\rightarrow \lim_{\langle Q\rangle}F
\]
is also a split epimorphism.
\end{lem}
\begin{proof}
Let $s\colon \lim_{\langle C(h)\rangle}F\to F(h)$  be a section as in the statement. We claim that the morphism 
\[
\begin{tikzcd}
\lim_{\langle Q\rangle} F \arrow[rr, "{(\res^{\langle Q\rangle}_{\langle C(h)\rangle},\Id)}"] &  & \lim_{\langle C(h)\rangle} F \times \lim_{\langle Q\rangle} F \arrow[rr, "{(s,\Id)}"] &  & F(h)\times \lim_{\langle Q\rangle} F
\end{tikzcd}
\]
induces a section as in the thesis of the theorem. Since $Q$ is compatible with the coatom ordering and $h$ is the least element in $\CP_{\prec \hat1}\setminus Q$, it follows that $Q\cup\{h\}$ is compatible with the recursive coatom ordering. By Lemma \ref{lem:Q_CL-shellable}, $\langle Q\cup \{h\}\rangle \cup \{\hat 1\}$ is a dual CL-shellable poset and then, by Lemma \ref{lema:Q_final_category}, its subposet with objects of codegrees $1$ and $2$ is final. Thus, it is enough to check that, for any $q\in Q$, $l\in \CP_{\le h}\cap \CP_{\le q}$ with $\cd(l)=2$, and $(x_q)_{q\in Q}\in \lim_{\langle Q\rangle} F\le \prod_{q\in Q} F(q)$, we have 
\[
F(l<q)(x_q)=F(l<h)(s(\res^{\langle Q\rangle}_{\langle C(h)\rangle} (x))).
\]
But such an element $l$ belongs to $C(h)$ by \eqref{equ:C(h)_h_cotaom}, and then the equation above  holds by the definition of the section $s$.
\end{proof}

\begin{lem}
\label{seccion_matching}
Let $(\CP,\ll)$ be a dual CL-shellable pair, $Q\subset \CP_{\prec \hat 1}$ be an initial segment of $\ll_{\hat1}$ and $F\colon \CP^{\op}\to R\Mod$ be a functor. Then the restriction morphism
\[
{\lim}_{\CP_{<\hat 1}}\NR{F}\to {\lim}_{\langle Q\rangle}\NR{F}
\]
is a degreewise split epimorphism.
\end{lem}
\begin{proof}
For any $i\ge 0$, we show that the restriction morphism at degree $i$,
\[
{\lim}_{\CP_{< \hat 1}}\NR{F}^i\to {\lim}_{\langle Q \rangle} \NR F^i,
\]
is a split epimorphism, and we proceed by induction on the degree the poset $\CP$. If $d(\CP)=1$, the result is trivial. Assume then that the statement is true for degree less than $n$ and that $\d(\CP)=n$. Set $Q_0=Q$ and, for $j\geq 1$, define $Q_j=Q_{j-1}\cup \{h_j\}$ with $h_j$ the least element in $\CP_{\prec \hat 1}\setminus Q_{j-1}$. We first show that the restriction 
\[
{\lim}_{\langle Q_j\rangle}\NR{F}^i\rightarrow {\lim}_{\langle Q_{j-1}\rangle}\NR{F}^i
\]
has a section $s_j$. By \cref{auxiliar_seccion}, it is enough to show that the composite
\[
\NR F^i(h_j)\rightarrow {\lim}_{\CP_{<h_j}}\NR{F}^i\rightarrow {\lim}_{\langle C(h_j)\rangle}\NR{F}^i
\]
admits a section. But, by \cref{prop:cocyl-factorisation}, the leftmost morphism is a split epimorphism,  and hence it suffices to show that the same holds for the rightmost morphism. By \textbf{CL1} in \cref{def:recursive_coatom_ordering} and \cref{lem:Q_CL-shellable}, $\langle C(h_j)\rangle\cup \{h_j\}$ is a CL-shellable poset of degree $\d(h_j)<\d(\CP)$, so the induction hypothesis gives the existence of such a section. If $\CP_{\prec 1}\setminus Q$ is finite with $m$ elements, then the theorem is proven by considering the section $s_m\circ \cdots \circ s_2\circ s_1$. Otherwise, if $x\in {\lim}_{\langle Q \rangle} \NR F^i$, we define an element in ${\lim}_{\CP_{< \hat 1}}\NR{F}^i$ by $p\mapsto (s_j\circ \cdots \circ s_1(x))(p)$ if $p\in Q_j$. Note that, if $p\in Q_j\cap Q_{k}$, then $(s_j\circ \cdots \circ s_1(x))(p)=(s_k\circ \cdots \circ s_1(x))(p)$ by construction. Thus, this map is well defined, and it is readily checked that it is a morphism in $R\Mod$ and a section.
\end{proof}


Given a functor $F\colon \CP^{\op}\to R\Mod$, we denote by $\ker \NR{F}^i\colon \CP^{\op}\to R\Mod$ the sub-functor of $\NR F^i\colon \CP^{\op}\to R\Mod$ defined on objects by 
\[
\ker \NR F^i(p)=\ker (\partial_p \colon\NR F^i (p)\to \NR F^{i+1}(p)).
\]

\begin{lem}
\label{CL/lem/aux_shellable}
Let $(\CP,\ll)$ be a dual CL-shellable pair, ${F\colon \CP^{\op}\to R\Mod}$ be a functor, and $i\geq 1$. If, for every coatom $h\in \CP_{\prec \hat{1}}$, the morphism
\[
\ker\NR F^{i-1}(h)\to \lim_{\langle C(h)\rangle}\ker \NR F^{i-1}
\]
is an epimorphism, then $H^i(\CP\setminus \{\hat1\};F)=0$.
\end{lem}
\begin{proof}
By \cref{prop:hig-lim-via-RF}, $H^i(\CP\setminus \{\hat1\};F)=H^i(\lim_{\CP\setminus\{\hat 1\}}\NR F)$, and hence we want to prove that the sequence
\[
\lim_{\CP\setminus\{\hat1\}}\NR F^{i-1}\overset\partial\longrightarrow \lim_{\CP\setminus\{\hat1\}}\NR F^i\overset\partial\longrightarrow \lim_{\CP\setminus\{\hat1\}}\NR F^{i+1}
\]
is exact. We write $\ll=\ll_{\hat{1}}$ for the linear order in $\CP_{\prec\hat{1}}$ and denote 
\[
\CP_{\prec\hat{1}}=h_0\ll h_1\ll \ldots \ll h_j\ll\ldots\quad\text{and}\quad\CP_j=\langle h_0,\dots ,h_j\rangle
\]
where, by \textbf{CL0}, either $0\leq j\leq m$ for some natural number $m$ or $j\geq 0$. So assume that $\partial X=0$ for an element 
\[
X=(x_j)\in \lim_{\CP\setminus\{\hat 1\}}\NR F^i\le \prod_{j}\NR F^i(h_j).
\]
We construct, for any $j$, an element 
\[
Y_j=(y_0,\ldots,y_j)\in\lim_{\CP_j} \NR F^{i-1}\le \prod_{k\le j}\NR F^{i-1}(h_k)
\]
such that $\partial Y_j=X_j=(x_0, \ldots,x_j)$. For $j=0$, as $\partial x_0=0$ and, by \cref{prop:cocyl-factorisation}, $H^i(\NR F(h_0))=0$, we have an element $y_0\in \NR F^{i-1}(h_0)$ such that $\partial y_0=x_0$. Now, assume that there exists $Y_{j-1}$ with $\partial Y_{j-1}=X_{j-1}$. Again, as $\partial x_{j}=0$, we find $\tilde y_{j}$ with $\partial \tilde y_j=x_j$. Consider the element $Z=(z_l)_{l\in C(h_j)}$ given by 
\[
z_l=\NR F^{i-1}(l < h_k)(y_k)-\NR F^{i-1}(l < h_j)(\tilde y_j),
\]
where, for every $l\in C(h_j)$, we have chosen $h_k\in \CP_{\prec \hat 1}$ with $h_k\ll h_j$ and $l<h_k$. Note that the definition of $z_l$ does not depend on the chosen elements $h_k$'s. In addition, we have
\begin{align*}
\partial z_l=&\NR F^{i-1}(l < h_k)(\partial y_k)-\NR F^{i-1}(l < h_j)(\partial \tilde y_j)\\
=&\NR F^{i-1}(l < h_k)(x_k)-\NR F^{i-1}(l < h_j)(x_j)
\end{align*}
and this is zero as $X$ belongs to the limit. Thus, $Z$ belongs to $\lim_{\langle C(h_j)\rangle}\ker\NR F^{i-1}$ and, by hypothesis, there exists $w\in \ker\NR F^{i-1}(h_j)$ with 
\[
\NR F^{i-1}(l < h_j)(w)=z_l
\]
for all $l\in C(h_j)$. We set $y_j=\tilde y_j + w$ and check that $Y_j\in\lim_{\CP_j} \NR F^{i-1}$:  By Lemma \ref{lem:Q_CL-shellable}, $\langle \{h_0,\ldots,h_j\}\rangle \cup \{\hat 1\}$ is a dual CL-shellable poset and then, by Lemma \ref{lema:Q_final_category}, its subposet with objects of codegrees $1$ and $2$ is final. Thus, it is enough to show that, for any $l\in \CP_{\le h_j}\cap \CP_{\le h_k}$  with $0\leq k\leq j-1$ and $\cd(l)=2$, we have
\[
\NR F^{i-1}(l<h_k)(y_k)=\NR F^{i-1}(l< h_j)(y_j).
\]
But, by Equation \eqref{equ:C(h)_h_cotaom}, the element $l$ belongs to $C(h_j)$, and thus we have
\[
\NR F^{i-1}(l< h_j)(\tilde y_j)+\NR F^{i-1}(l< h_j)(w)=\NR F^{i-1}(l< h_j)(\tilde y_j)+z_l=\NR F^{i-1}(l < h_k)(y_k).
\]
In addition, by construction we have $\partial y_j =\partial\tilde{y_j}+\partial w=x_j$, and hence $\partial Y_j=X_j$. Finally, the element constructed inductively over all indexes $j$, i.e., either $Y=Y_m=(y_0,\ldots,y_m)$ or $Y=(y_0,\ldots,y_j,\ldots)$, belongs to $\lim_{\CP\setminus\{\hat1\}}\NR F^{i-1}$ and satisfies that $\partial Y =X$. We are done.
\end{proof}


\begin{proof}[Proof of {\cref{main_theorem}}]
We proceed by induction on $i$. For $i=1$, by \cref{CL/lem/aux_shellable}, it is enough to check that, for every coatom $h$, the morphism
\[
\ker \NR F^0(h)\to \lim_{\langle C(h)\rangle}\ker\NR F^0
\]
is an epimorphism.  But $\ker \NR F^0$ and $F$ are naturally isomorphic, and hence, this condition is exactly the stability property at $i=1$. Next, assume that the theorem is true for $1\le j<i$. Again, by applying \cref{CL/lem/aux_shellable}, we have to check that, for every coatom $h$ the morphism
\[
\ker \NR F^{i-1}(h)\to \lim_{\langle C(h)\rangle}\ker\NR F^{i-1}
\]
is an epimorphism. Consider the following commutative diagram
\[
\begin{tikzcd}
\ker \NR F^{i-1}(h) \arrow[rr]     &                               & \lim_{\langle C(h)\rangle}\ker \NR F^{i-1}      \\
\NR F^{i-2}(h) \arrow[r] \arrow[u] & \lim_{\CP_{< h}} \NR F^{i-2} \arrow[r] & \lim_{\langle C(h)\rangle}\NR F^{i-2}, \arrow[u]
\end{tikzcd}
\]
where the vertical morphisms are the respective differentials. We prove that the top arrow is an epimorphism by proving that every other morphism in the diagram is so. By exactness of $\NR F(h)$ and because $i\geq 2$, the leftmost vertical map is an epimorphism. The bottom arrow is onto because $\NR F$ is a fibrant functor and because of \cref{seccion_matching}. Thus, we are left with the rightmost vertical arrow: By \cref{lem:Q_CL-shellable}, the poset $\CR=\langle C(h)\rangle\cup\{h\}$ is a dual CL-shellable poset with recursive coatom ordering $\ll'$. Moreover, $(\CR,\ll',F\vert_{\CR})$ has the stability property at $i-1$: Let $c'$ be an unrefinable $h$-chain of length $i-1$. Then $c=c'\prec \hat 1$ is an unrefinable $\hat 1$-chain in $\CP$ of length $i$, and thus, by hypothesis, the following map is surjective,
\[
F(c'_0)=F(c_0)\to \lim_{\langle C(c) \rangle} F=\lim_{\langle C'(c') \rangle} F \vert_{\CR},
\]
where $C(c)$ and $C'(c')$ are taken with respect to $\ll$ and $\ll'$ respectively and they are equal because the length of $c'$ is at least $2$, see proof of \cref{lem:Q_CL-shellable}. Thus, by induction, $H^{i-1}(\langle C(h)\rangle; F\vert_{\CR})=0$ and hence the rightmost arrow in the diagram above is surjective.
\end{proof}

\begin{exm}
\label{ex:constant-functor}
Let $\CP$ be a dual CL-shellable poset, $M\in R\Mod$, and let $F\colon \CP^\op\to R\Mod$ be the extension by zero of the constant functor $\underline{M}\colon \overline{\CP}^\op\to R\Mod$. By \cite[Theorem 5.8 and 5.9]{BWnonpure}, the order complex $K$ of $\overline\CP^\op$ has the homotopy type of a wedge of spheres. The stability property can detect dimensions where no sphere shows up. Let $1\le i\le d(\CP)-1$: If for every unrefinable $\hat 1$-chain $c$, the set $C(c)$ does not contain any atom of $\CP$, then it is straightforward $F$ has the stability property at $i$ (see Lemmas \ref{lema:Q_final_category} and \ref{lem:Q_CL-shellable}). Then, by \cref{main_theorem}, we have
\[
H^i(K;M)\cong H^i(\CP\setminus\{\hat 1\}; F)=0.
\]
In particular, if $\CP$ is pure, the previous equation holds for all $i=1,\dots, \d(\CP)-3$.
\end{exm}

\section{Non-vanishing (co)homology of stable functors}\label{sec:non-vanishing_stable_functor}
In this section, we describe the top (co)homology group of a functor $F$ with $F(\hat 0)=0$ over a finite pure dual CL-shellable poset. 
Note that, by \cref{rmk:stability_property_cases}, such a functor is $\d(\CP)-1$ stable, so that 
$\d(\CP)-2$ is the highest degree for which $H^*(\CP\setminus\{\hat 1\};F)$ can possibly be different from zero. In addition, $F$ is $\d(\CP)-2$ stable if and only if the following map is surjective, 
\[
F(c_0)\to \bigoplus_{a\in C(c)} F(a).
\]
In \cref{thm:exact_sequence_for_atoms} below, we drop this condition of stability and we describe the possibly non-trivial cohomology at degree $\d(\CP)-2$. We start describing the top (co)homology group of atomic functors over CL-shellable posets. 

\begin{defn}
Let $\CP$ be a poset, $p\in \CP$ and $M\in R\Mod$. The contravariant (resp. covariant) \emph{atomic functor} at $p$ with value $M$ is the functor $A(p,M)\colon \CP^\op\to R\Mod$ (resp. ${A(p,M)\colon \CP\to R\Mod}$) defined on objects by
\[
A(p,M)(q)=\begin{cases}
M&\text{if }p=q,\\
0&\text{otherwise.}
\end{cases}
\]
\end{defn}

Note that in the previous definition we have abused notation and denoted by $A(p,M)$ both contravariant and covariant atomic functors. Although the Möbius function makes sense for locally finite posets, here we restrict ourselves to finite posets for simplicity.

\begin{defn}
For a finite poset $\CP$ and elements $p\leq q$ we define 
\[
\mu(p,q)=\begin{cases} 
1&\text{ if $p=q$,}\\
-\sum_{p<r\leq q} \mu(r,q)&\text{otherwise.}
\end{cases}
\]
\end{defn}

\begin{prop}\label{prop:ho_and_coho_of_atomic_functor}
Let $\CP$ be a finite pure dual CL-shellable poset, $p\in \CP\setminus \{\hat{1}\}$, and $M\in R\Mod$. Then 
\[
H^i(\CP\setminus \{\hat{1}\};A(p,M))\cong H_i(\CP\setminus \{\hat{1}\};A(p,M))\cong \begin{cases} M^{|\mu(p,\hat 1)|}&\text{if }i=\cd(p)-1,\\
0&\text{otherwise.}
\end{cases}
\]
\end{prop}
\begin{proof}
Consider the short exact sequence of contravariant functors,
\[
0\to A(p,M)\to B\to C\to 0,
\]
and the short exact sequence of covariant functors,
\[
0\to D \to E \to A(p,M) \to 0,
\]
where $B(q)=E(q)=M$ if $p\leq q$ and they are $0$ otherwise, and $C$ and $D$ are defined, respectively, as the appropriate cokernel or kernel. Then $B$ and $E$ are free functors and hence $\lim$-acyclic and $\colim$-acyclic respectively (also by \cite[Theorems $B$ and $B^*$]{CD23} or Theorems~\ref{main_theorem} and \ref{dual/main_theorem}). From the associated long exact sequences we deduce that, for $i\geq 2$,
\begin{align}
H^i(\CP\setminus \{\hat{1}\};A(p,M))&\cong H^{i-1}(\CP\setminus \{\hat{1}\};C)\cong H^{i-1}(|(\CP\setminus\{\hat{1}\})_{>p}|; M)\text{ and }\label{equ:less_atomic}\\
H_i(\CP\setminus \{\hat{1}\};A(p,M))&\cong H_{i-1}(\CP\setminus \{\hat{1}\};D)\cong H_{i-1}(|(\CP\setminus\{\hat{1}\})_{>p}|; M)\text{,}\nonumber
\end{align}
where in the last isomorphism of each row we have used that $\CP_{>p}$ is a right ideal of $\CP$, that $C$ and $D$ vanish outside this ideal, and that $C$ and $D$ are constant on this ideal. Since $\CP$ is a finite pure dual CL-shellable poset we have, by \cite[Lemma 5.6]{BWnonpure}, that so it is $\CP_{\ge p}$. Thus, by \cite[Theorems 5.9]{BWnonpure}, its integral homology is
\[
\widetilde H_j(|(\CP\setminus\{\hat{1}\})_{>p}|; \ZZ)=\begin{cases} \ZZ^{\delta}&\text{if }j=\cd(p)+1,\\
0&\text{otherwise.}
\end{cases}
\]
where $\delta$ is the number of maximal falling chains of length $\cd(p)+1$ in $\CP_{\geq p}$ (with the notion of length in this work). Thus, by \cite[Proposition 5.7]{BWnonpure}, $\delta$ is exactly the absolute value $|\mu(p,\hat 1)|$. If $\cd(p)\geq 3$ and $i\geq 2$, we are done by Equation \eqref{equ:less_atomic} and universal coefficient theorem. Otherwise, a short explicit computation suffices. 
\end{proof}

\begin{thm}\label{thm:exact_sequence_for_atoms}
Let $(\CP,\ll)$ be a dual CL-shellable pair with $\CP$ pure and finite and with $\d(\CP)\geq 3$ and consider a functor $F\colon \CP^\op\to R\Mod$ with $F(\hat 0)=0$. Assume that there exists a subset $\CA$ of atoms of $\CP$ such that for every unrefinable $\hat 1 $-chain $c$ of length $\d(\CP)-2$, we have an epimorphism
\[
F(c_0)\to \bigoplus_{a\in C(c)\setminus \CA} F(a).
\]
Then there is an exact sequence with central arrow the connecting homomorphism,
\[
0\to H^{\d(\CP)-3}(\CP\setminus\{\hat 1\};F)\to H^{\d(\CP)-3}(\CP\setminus \CA;F) \to \bigoplus_{a\in \CA} F(a)^{|\mu(a,\hat 1)|}\to H^{\d(\CP)-2}(\CP\setminus\{\hat 1\};F)\to 0.
\]
\end{thm}

\begin{proof}
Consider the short exact sequence of contravariant functors 
\[
0\to E\to F\to G\to 0,
\]
where $E=\bigoplus_{a\in \CA} A(a,F(a))$, $G$ is the quotient of $F$ by $E$. Since $F(\hat 0)=0$, by \cref{fib-rep-F(0)=0}, it is equivalent to compute the higher limits of the involved functor over $\overline\CP$.  Consider as well as the following portion of the associated long exact sequence, where we write  $n=\d(\CP)$ for brevity,
\[
H^{n-3}(\overline\CP;E)\to H^{n-3}(\overline\CP;F)\to H^{n-3}(\overline\CP;G)\to H^{n-2}(\overline\CP;E)\to H^{n-2}(\overline\CP;F)\to H^{n-2}(\overline\CP;G).
\]
Because $\overline\CP\setminus \CA$ is a right ideal of $\overline\CP$ and $G$ vanishes on $\CA$, we have the isomorphism
\[
H^{n-3}(\overline\CP;G)\cong H^{n-3}(\overline\CP\setminus \CA;G)=H^{n-3}(\overline\CP\setminus \CA;F).
\]
Moreover, the hypotheses imply that $G$ has the stability property at $\d(\CP)-2$ and \cref{main_theorem} shows then that the rightmost term is zero. The description of $H^{n-3}(\overline\CP;E)$ and $H^{n-2}(\overline\CP;E)$ follows from \cref{prop:ho_and_coho_of_atomic_functor} as $\cd(a)=\d(\CP)-1$ for every $a\in \CA$.
\end{proof}

\begin{cor}\label{cor:exact_sequence_for_atoms_and_stable_functor}
Assume the hypotheses of Theorem \ref{thm:exact_sequence_for_atoms}, that $\d(\CP)\geq 4$, and that $F$ has the stability property at $\d(\CP)-3$. Then we have a short exact sequence
\[
0 \to H^{\d(\CP)-3}(\CP\setminus \CA;F) \to \bigoplus_{a\in \CA} F(a)^{|\mu(a,\hat 1)|}\to H^{\d(\CP)-2}(\CP\setminus\{\hat 1\};F)\to 0.
\]
\end{cor}
\begin{proof}
This is immediate from \cref{thm:exact_sequence_for_atoms} and \cref{main_theorem}.
\end{proof}

We state without proof the dual results.

\begin{thm}\label{thm:exact_sequence_for_atoms_covariant_colimit}
Let $(\CP,\ll)$ be a dual CL-shellable pair with $\CP$ pure and finite and with $\d(\CP)\geq 3$ and consider a functor $F\colon \CP\to R\Mod$ with $F(\hat 0)=0$. Assume that there exists a subset $\CA$ of atoms of $\CP$ such that for every unrefinable $\hat 1 $-chain $c$ of length $\d(\CP)-2$, we have a monomorphism 
\[
\bigoplus_{a\in C(c)\setminus \CA} F(a)\to F(c_0).
\]
Then there is an exact sequence with central arrow the connecting homomorphism,
\[
0\to H_{\d(\CP)-2}(\CP\setminus \{\hat 1\};F)\to \bigoplus_{a\in \CA} F(a)^{|\mu(a,\hat 1)|}\to H_{\d(\CP)-3}(\CP\setminus \CA;F)\to H_{\d(\CP)-3}(\CP\setminus\{\hat 1\};F)\to 0.
\]
\end{thm}

\begin{cor}\label{cor:exact_sequence_for_atoms_and_stable_functor_covariant_colimit}
Assume the hypotheses of Theorem \ref{thm:exact_sequence_for_atoms_covariant_colimit}, that $\d(\CP)\geq 4$, and that $F$ has the stability property at $\d(\CP)-3$. Then we have a short exact sequence
\[
0\to H_{\d(\CP)-2}(\CP\setminus \{\hat 1\};F)\to \bigoplus_{a\in \CA} F(a)^{|\mu(a,\hat 1)|}\to H_{\d(\CP)-3}(\CP\setminus \CA;F)\to 0.
\]
\end{cor}

\begin{rmk}
Note that we can take $\CA$ equal to all the atoms of $\CP$ in Theorems \ref{thm:exact_sequence_for_atoms} and \ref{thm:exact_sequence_for_atoms_covariant_colimit} as well as in Corollaries \ref{cor:exact_sequence_for_atoms_and_stable_functor} and \ref{cor:exact_sequence_for_atoms_and_stable_functor_covariant_colimit}.
\end{rmk}

\section{Applications}

\subsection{Mackey functors} In this subsection, we show how the vanishing result for Mackey functors over a poset $\CP$ in \cite[Theorem C*]{CD23} can be improved when $\CP$ is dual CL-shellable. For the sake of completeness, we include below the necessary definitions.

\begin{defn}
\label{def:G-linearCC}
Let $G\colon \CP^\op \to R\Mod$ be a functor from a poset $\CP$ and $i\in \CP$. We say that $\alpha\in \End_R(G(i))$  ($\alpha\in \Aut_R(G(i))$) is \emph{$G$-linear} if for all $j<i$
\[
G(j<i)\circ\alpha = \beta \circ G(j<i)\]
for some $\beta\in \End_R(G(j))$ ($\beta\in \Aut_R(G(j))$). We denote by $\End^G_R(i)$ ($\Aut^G_R(i)$) the submonoid (subgroup) of $G$-linear endomorphisms (automorphisms) of $G(i)$.
\end{defn}

\begin{defn}\label{def:weakMackeyfunctorAb_Pop} 
Let $\CP$ be a filtered poset and let $G\colon \CP^\op\to R\Mod$ a contravariant functor. We say that $G$ is a \emph{weak Mackey functor} if for all $j<i$  there exists a morphism in $R\Mod$, $F(j<i)\colon G(j)\to G(i)$, such that $G(j<i)\circ F(j<i)=\alpha(i,j)$ with $\alpha(i,j)\in \End^G_{R}(j)$, and, for $k<i$, $j\not< k$, 
\[
\ker_G(k)\subseteq \ker(G(j<i)\circ F(k<i)),
\]
where $\ker_G(k)=\bigcap_{l<k}\ker(G(l<k))$. Moreover, if $\CQ$ is a subposet of $\CP$, we say that $G$ has a \emph{quasi-unit in  $\CQ$} if $\alpha(i,j)\in \Aut^G_{R}(j)$ for all $j<i$, with  $i,j\in \CQ$. If $\CQ=\CP$, we just say that $G$ has a \emph{quasi-unit}.
\end{defn}

While in \cite[Theorem C*]{CD23} it is proven that a weak Mackey functor $G\colon \CP^\op\to R\Mod$  with a quasi-unit is acyclic, i.e., $
H^{i}(\CP\setminus \{\hat 1\};F)=0$ for all $i\geq 1$, here we obtain the following refinement.

\begin{thm}
\label{thm:Mackey}
Let $\CP$ be a dual CL-shellable poset with recursive coatom ordering $\ll$, let $1\le i\le \d(\CP)-1$, and let $G\colon \CP^\op\to R\Mod$ be a weak Mackey functor such that for every unrefinable $\hat 1$-chain $c$ of length $i$, the functor $G$ has a quasi-unit in $\{c_0\}\cup \langle C(c)\rangle$. Then
\[
H^i(\CP\setminus \{\hat 1\};G)=0.
\]
\end{thm}
\begin{proof}
By Theorem \ref{main_theorem}, it is enough to check that $(\CP,\ll,G)$ has the stability property at $i$. If $c$ is as in the statement then, by hypothesis, the restriction of $G$ to $\CR=\{c_0\}\cup\langle C(c)\rangle$ is a weak Mackey functor with a quasi-unit. Thus, by \cite[Theorem C*, Theorem B*]{CD23}, this restriction satisfies \cite[Equation (12)]{CD23} for $i=c_0$, i.e.,
\[
G(c_0)\to \lim_{\CR_{<c_0}} G = \lim_{\langle C(c) \rangle} G
\]
is surjective. We are done.
\end{proof}

\subsection{Hyperplane arrangement.}

Let $V$ be a finite-dimensional vector space over a field $\KK$, and let $H$ be a numerable set of linear hyperplanes in $V$. The \emph{arrangement lattice} $\CL(H)$ consists of all possible intersections of hyperplanes in $H$ ordered by the inclusion relation. Then $\CL(H)$ is a pure bounded lattice by considering
\[
\hat 0 = \bigcap_{h\in H}h,\qquad \hat 1 = V,\qquad \d(x)=\dim(x)-\dim(\hat0),
\]
\[
\quad x\meet y=x\cap y,\quad x\join y=\bigcap\{z\in \CL(H)\mid x\cup y\subset z\}.
\]  
It is straightforward that, for $x\le y$ in $\CL(H)$, $x\prec y$ if and only if $\d(y)=\d(x)+1$. In the case of $H$ being finite, $\CL(H)$ is a geometric lattice \cite[Subsection 1.2]{Everitt2022}, and this implies that it is dual CL-shellable \cite{Bjoerner1980,Bjoerner1983}. Here, we prove that $\CL(H)$ is dual CL-shellable even if $H$ is not finite by providing a recursive coatom ordering, see \cref{lem:recursive_coatom_ordering_L}. For $x\in \CL(H)$ and a subset $I\subseteq \CL(H)_{\prec x}$, via annihilators we may think of the elements of $I$ as $1$-dimensional subspaces of the dual vector space $x^*$, and  we denote as $\langle I\rangle_\KK$ the vector subspace they span in $x^*$. Moreover, we have the following properties,
\begin{enumerate}
\item[(a)] $I$ is a basis of $\langle \CL(H)_{\prec x}\rangle_\KK$ if and only if it is minimal subject to $\bigcap_{h\in I} h=\hat 0$, 
\item[(b)] $\dim(\langle \CL(H)_{\prec x}\rangle_\KK)=\d(x)$, and
\item[(c)] $\langle I\rangle_{\CL(H)}= \langle I\rangle_{\CL(H)_{\le x}}$.
\end{enumerate}

\begin{lem}
\label{lem:recursive_coatom_ordering_L}
Let $V$ be a finite-dimensional vector space and $\CH$ be a numerable set of hyperplanes of $V$. Then there exists a recursive coatom ordering $\ll$ for $\CL(H)$ satisfying the additional condition
\begin{description}
\item[CL3] For every unrefinable $\hat 1$-chain $c$, there exists an initial segment of $\ll_c$ that is a basis of $\langle \CL(H)_{\prec x}\rangle_\KK$.

\end{description}
\end{lem}
\begin{proof}
Note that \textbf{CL2} holds by choosing $h''=h$ and $q=h\cap h'$ in that definition. Next, we define by induction on the length of the $\hat1$-chain $c$ a linear order $\ll_c$ that verifies \textbf{CL0,1,3}. For length $1$, i.e., for $c=(\hat1)$, \textbf{CL1} is trivial since $C(c)=\varnothing$. Let $B$ be a basis of $\langle H\rangle_\KK$ contained in $H$. Since $B$ is finite and $H$ numerable, we define $\ll_c$ to be any well order for $H$ verifying \textbf{CL0} and the following property:
\begin{equation}\label{equ:order_for_CL3}
x\in B\mbox{ and } y\in H\setminus B\mbox{, then } x\ll_c y.
\end{equation}
Now, assume that $\ll_c$ is defined for every unrefinable chain of length less than $i$, that it satisfies \textbf{CL0,1,3}, and let $c$ be an unrefinable chain of length $i$. Let $c'=c_1\prec c_2\prec \dots \prec \hat 1$, $B'$ be the initial segment given by \textbf{CL3} applied to $c'$, and consider the subset
\[
\hat B= \{b'\cap c_0\mid b'\in B', b'\ll_{c'} c_0\}.
\]
Then we have $\hat B\subseteq C(c)$ by the definition of $C(c)$. Moreover, if $c_0\notin B'$, then $\hat B$ contains a basis $B$ of $\langle \CL(H)_{\prec c_0}\rangle_\KK$ by property (a) above. Otherwise, $c_0\in B'$, $\hat B=C(c)$, and this latter set is linearly independent as $B'$ is so. Hence, we may extend $C(c)$ to a basis $B\subseteq \CL(H)_{\prec c_0}$. In any case, we can proceed in a similar way to \eqref{equ:order_for_CL3} and construct a well order $\ll_c$ of $\CL(H)_{\prec c_0}$ that satisfies \textbf{CL0} and such that both $B$ and $C(c)$ are initial segments.
\end{proof}

For any $j\geq 0$, consider the covariant functor given by $j$-th exterior power,
\[
\xymatrix@R=8pt{
\Lambda^j\colon \CL(H)\ar[r] & \KK\Mod\\
W\ar@{|->}[r]& \Lambda^j(W)\\
W'\leq W\ar@{|->}[r]& \Lambda^j(W')\leq \Lambda^j(W),
}
\]
where $\Lambda^j(W)$ is the $j$-th exterior power on the vector space $W$. Now, employing the co-stability property of \cref{sec:vanishing} and the non-vanishing results of \cref{sec:non-vanishing_stable_functor}, we produce new proofs for Theorem~9 and part of Theorem~11 in \cite{Everitt2019}. As remarked in the introduction, our results apply to arrangements with numerable number of hyperplanes.

\begin{thm}
\label{thm:hyperplane arrangement}
Let $V$ be a finite-dimensional vector space, $H$ be a numerable set of hyperplanes of $V$, and $\CL(H)$ the arrangement lattice of $H$. Then $d_{i,j}=\dim H_i(\CL(H)\setminus\{V\};\Lambda^j(-))$ is zero outside of the set $A\cup B$, where $A$ is the segment
\[
A=\{(0,j)|0\leq j\leq \dim(V)-1\}
\]
and $B$ is the truncated band
\[
B=\{(i,j)\text{ such that } 1\leq i\leq \d(\CL(H))-2 \text{ and }\d(\CL(H))-1\leq i+j\leq \dim(V)-1\}.
\]
Moreover, $d_{0,j}=\begin{pmatrix} \dim(V)\\j \end{pmatrix}$ if $0\leq j\leq \d(\CL(H))-2$ and $d_{0,\dim(V)-1}=\#H$. In addition, if $H$ is finite, $\dim(\hat 0)=0$, $1\leq i\leq \d(\CL(H))-2$, and $i+j=\d(\CL(H))-1$, then $d_{i,j}$ equals
\[
(-1)^{i+1}\begin{pmatrix} \dim(V)\\j \end{pmatrix} + \sum_{d=\d(\CL(H))-1-i}^{\d(\CL(H))-1} (-1)^{d-\d(\CL(H))+1+i} \begin{pmatrix} d\\j \end{pmatrix} \sum_{\substack{x\in \CL(H)\\\d(x)=d}} |\mu(x,V)|.
\]
\end{thm}
\begin{rmk}
For $H$ finite and $\hat 0=0$, the formula for $d_{0,\dim(V)-1}$ simplifies the formula in the second case in \cite[Theorem 9]{Everitt2019}. The formulas for $d_{0,j}$ and $d_{i,j}$ are alternative explicit formulations of the first and third cases in the same theorem.
\end{rmk}

\begin{proof}
Let $\CL=\CL(H)$. Note that if $i>\d(\CL)-2$, we have  $d_{i,j}=0$ by dimensional reasons and thus we assume in the rest of the proof that $0\leq i\leq  \d(\CL)-2$. In addition, \begin{equation}\label{equ:Lambdaj(x)=0_dim(x)<j}
\Lambda^j(x)=0\text{ if }\dim(x)<j,    
\end{equation}  
so that $d_{i,j}=0$ if $i+j\geq \dim(V)$ by dimensional reasons again. 

Next, we prove that, for $1\leq i\leq \d(\CL)-2$ and $i+j\leq \d(\CL)-2$, we also have that $d_{i,j}=0$. For this is enough, by Theorem \ref{dual/main_theorem}, to show that $(\CL,\ll,\Lambda^j)$ has the co-stability property at $i+1$, where $\ll$ is a recursive coatom ordering as in Lemma \ref{lem:recursive_coatom_ordering_L}. Thus, for any unrefinable $\hat1$-chain $c$ of length $i+1$ in $\CL$, we need to show that 
\begin{equation}\label{equ:colimLambdal_main}
\colim_{\langle C(c) \rangle} \Lambda^j \to \Lambda^j(c_0)\text{ is injective.}
\end{equation}
Note that $dim(c_0)=dim(V)-i$. We proceed in several steps.

\textbf{Step 1: } For a vector space $W$ with basis $\{w_1,\ldots,w_n\}$, $m\leq n$, and a lattice $\CL$ satisfying
\[
\{W_1,\ldots,W_m\}\subseteq \CL\subseteq \CL(\{\text{hyperplanes of }W\}),
\]
where $W_k$ is the hyperplane generated by the vectors $w_\cdot$'s but $w_k$, the morphism
\begin{equation}\label{equ:colimLambdal_general}
\colim_{\langle W_1,\ldots,W_m\rangle_\CL } \Lambda^j \to \Lambda^j(W)
\end{equation}
is injective and, if $j\leq m-1$, it is an isomorphism. Recall that $\Lambda^j(W)$ has basis the wedges 
\[
\{w_{k_1}\wedge \ldots \wedge w_{k_j}\text{ with }1\leq k_1<\cdots<k_j\leq n\},
\]
and the domain of \eqref{equ:colimLambdal_general} is a quotient of the $\KK$-vector space generated by the wedges
\[
\bigcup_{k=1}^m \{w_{k_1}\wedge \ldots \wedge w_{k_j}\text{ with }1\leq k_1<\cdots<k_j\leq n\text{, }k_* \neq k\}.
\]
If $j\geq n$, the claim about injectivity is clear as the colimit vanishes. If $j=n-1$, there are no relations in the colimit and injectivity is also clear. For $j\leq n-2$, if a wedge $w_{k_1}\wedge \ldots \wedge w_{k_l}$ appears as $w\in \Lambda^l(W_k)$ and $w'\in \Lambda^l(W_{k'})$ with $1\leq k,k'\leq m$, i.e., it satisfies that $k_* \neq k$ and $k_* \neq k'$, then it also belongs to $\Lambda^l(W_k\cap W_{k'})\neq 0$, where $W_k\cap W_{k'}\in \langle W_1,\ldots,W_m\rangle_\CL$. Thus, we have $w=w'$ in the colimit and injectivity follows. If $j\leq m-1$, the two sets of wedges above are identical and surjectivity follows.

\textbf{Step 2: } Equation \eqref{equ:colimLambdal_main} holds if $C(c)$ is contained in an initial segment of $\ll_{c}$ that is a basis $B$ of the vector space $\langle \CL_{\prec c_0}\rangle_\KK\leq c_0^*$. Write $C(c)=\{h_1,\ldots,h_m\}$, $B=\{h_1,\ldots,h_r\}$, and 
\[
B_0=\{h_1,\ldots,h_r,h'_{r+1},\ldots,h'_n\},
\]
where $m\leq r\leq n$ and the latter set is a basis of $c_0^*$. For the dual basis $B_0^*=\{w_1,\ldots,w_n\}$ of $c_0$, we have that $h_k$ is generated by all the $w_\cdot$'s but $w_k$ and hence, by the injectivity part of Step~1 applied with $\CL=\CL(H)_{\le c_0}$ and property (c) above, Equation  \eqref{equ:colimLambdal_main} follows.

\textbf{Step 3: } Equation \eqref{equ:colimLambdal_main} holds if $C(c)$ contains an initial segment that is a basis $B$ of the vector space $\langle \CL_{\prec c_0}\rangle_\KK$. Write $B=\{h_1,\ldots,h_m\}$, $C(c)=\{h_1,\ldots,h_r\}$, and 
\[
B_0=\{h_1,\ldots,h_m,h'_{m+1},\ldots,h'_n\},
\]
where $m\leq r\leq n$ and the latter set is a basis of $c_0^*$. Next, consider the following composition,
\[
\colim_{\langle h_1,\ldots,h_m\rangle_\CL } \Lambda^j\to \colim_{\langle h_1,\ldots,h_r\rangle_\CL } \Lambda^j=\colim_{\langle C(c)\rangle_\CL} \Lambda^j\to \Lambda^j(c_0),
\]
where $\CL=\CL(H)_{\le c_0}$. Again by the injectivity part of Step 1, this composition is injective. Thus,  it suffices to show that the left hand-side morphism is surjective. To that aim, we consider a class 
\[
[\oplus_{l=1}^r x_l]\in\colim_{\langle C(c)\rangle_\CL} \Lambda^j,
\]
where $x_l\in \Lambda^j(h_l)$, and show that we may take $x_l=0$ if $m<l\leq r$. In fact, in that situation, 
\[
\bigcap_{k=1,\ldots,m} h_l\cap h_k=h_l\cap \bigcap_{k=1,\ldots,m} h_k=h_l\cap \hat0=\hat0,
\]
and thus, by (a), we may assume that the hyperplanes  $\{h_l\cap h_k\}_{k=1,\ldots,m-1}$ form a basis of the vector space $\langle \CL(H)_{\prec h_l}\rangle_\KK$ whose dimension is, by (b), equal to $\d(h_l)=\d(c_0)-1=m-1$. Thus, if $j\leq m-2$, by the isomorphism part of Step 1, we have a surjection,
\[
\bigoplus_{k=1\ldots m-1} \Lambda^j(h_l\cap h_k) \to \Lambda^j(h_l).
\]
In particular, $x_l=\sum_{k=1,\ldots,m-1} x_{l,k}$ with $x_{l,k}\in \Lambda^j(h_l\cap h_k)$ and 
\[
[x_l]=[\sum_{k=1,\ldots,m-1} x_{l,k}]=[\sum_{k=1,\ldots,m-1} x'_{l,k}],
\]
where $x'_{l,k}$ is the image of $x_{l,k}$ by $\Lambda^j(h_l\cap h_k)\to \Lambda^j(h_k)$ for $k=1,\ldots,m-1$. To finish, note that by property (b) above, we have
\[
m=\dim(\langle \CL_{\prec c_0}\rangle_\KK)=d(c_0)=\d(\CL)-i,
\]
and hence the condition $j\leq m-2$ is identical to $i+j\leq \d(\CL)-2$. 

Next, we deal with the case $i=0$ and $0\leq j\leq \d(\CL)-2$ of the statement. Write ${H=\{h_1,\ldots,h_n\}}$ (possibly $n=\infty$) and $B=\{h_1,\ldots,h_m\}$ with $m=\d(\CL)\leq n$ and where $B$ is a basis of $\langle \CL_{\prec \hat 1}\rangle_\KK=\langle H\rangle_\KK\leq V^*$. Then the following composition 
\[
\colim_{\langle h_1,\ldots,h_m\rangle_\CL } \Lambda^j\to \colim_{\CL } \Lambda^j=H_0(\CL\setminus\{V\};\Lambda^j)\to \Lambda^j(V)
\]
is injective by Step 1 and bijective if $j\leq \d(\CL)-1$ by Step 1. The leftmost map is surjective if $j\leq \d(\CL)-2$ by the argument of Step 3 and hence the result holds. 

For the case $i=0$ and $j=\dim(V)-1$, we have that $\Lambda^j(x)=\KK$ for $x\in H$ and $\Lambda^j(x)=0$ otherwise by \cref{equ:Lambdaj(x)=0_dim(x)<j}. The result follows.

Finally, we consider the case $H$ finite, $\dim(\hat 0)=0$, $1\leq i\leq \d(\CL)-2$, and $i+j=\d(\CL)-1$. Note that in this case we have $\d(\CL)=\dim(V)$ and, in general, $\d(x)=\dim(x)$ for all $x\in \CL$. In addition, by \eqref{equ:Lambdaj(x)=0_dim(x)<j}, we have that $\Lambda^j(x)=0$ if $\d(x)<\d(\CL)-1-i$. Thus, for the right ideal $\CL'=\CL_{\geq \d(\CL)-1-i}$ we have that $\d(\CL')=i+2$ and that
\[
H_*(\CL\setminus\{V\};\Lambda^j(-))\cong H_*(\CL'\setminus\{V\};\Lambda^j(-))\text{ for all $*\geq 0$}.
\]
Let $\CA$ equal all the atoms of $\CL'$ and set $\CL''=\CL'\setminus \CA=\CL_{\geq \d(\CL)-i}$ which has $\d(\CL'')=i+1$. Then, if $i=1$, \cref{thm:exact_sequence_for_atoms_covariant_colimit} gives an exact sequence
\[
0\to H_1(\CL\setminus \{V\};\Lambda^j)\to \bigoplus_{\substack{x\in \CL\\\d(x)=\d(\CL)-1-i}} \Lambda^j(x)^{|\mu(x,V)|}\to  \bigoplus_{\substack{x\in \CL\\\d(x)=\d(\CL)-i}} \Lambda^j(x)^{|\mu(x,V)|}\to  \colim_{\CL } \Lambda^j\to 0,
\]
and, if $i\geq 2$, the acyclic case above and \cref{cor:exact_sequence_for_atoms_and_stable_functor_covariant_colimit}  give a short exact sequence
\[
0\to H_i(\CL\setminus \{V\};\Lambda^j)\to \bigoplus_{\substack{x\in \CL\\\d(x)=\d(\CL)-1-i}} \Lambda^j(x)^{|\mu(x,V)|}\to H_{i-1}(\CL''\setminus\{V\};\Lambda^j)\to 0.
\]
In the former case, i.e., $i=1$, we have $j=\d(\CL)-2$ and, by the previous cases, $d_{0,j}= \begin{pmatrix} \dim(V)\\j  \end{pmatrix}$ and $\colim_{\CL } \Lambda^j\cong \Lambda^j(V)$. Furthermore, Step 1 is valid under the weaker condition  (with the notation there) that $\CL$ is a poset (instead of a lattice) containing $\{W_1,\ldots,W_m\}$ and their pairwise intersections. As Steps 2 and 3 are consequences of Step 1, we deduce that, for the truncation posets $\CL_{i'}=\CL_{\geq \d(\CL)-i'}$ with $3\leq i'\leq i$, Equation \eqref{equ:colimLambdal_main} for length $i'-1$ is valid and hence that $H_{i'-2}(\CL_{i'};\Lambda^j)=0$. Then we may induct on $i$ to finish the proof.
\end{proof}

\subsection{Homology decompositions for Bianchi groups.} In this subsection, for each Bianchi group $\Gamma_d=\PSL_2(\mathcal{O}_d)$, where $\mathcal{O}_d$ is the ring of integers of the quadratic field $\QQ(\sqrt{-d})$ and $d=\Bianchids$, we define a poset of subgroups $\CP_d$ containing the trivial subgroup and with $\Gamma_d\cong\colim_{U\in \CP_d} U$. Then we show that the fiber $F_d$ of the following fibration, see \cite[Theorem 5.1]{DiazRamos2009} and \cite[Theorem E]{CD23}, is contractible, proving \cref{thm:integral_homology_decomposition_Bianchi},
\begin{equation}\label{equ:fibration_to_BGamma}
F_d\to \hocolim_{U\in \CP_d} B(U)\to B(\Gamma_d).
\end{equation}
In fact, $F$ is simply connected and we have  $H_n(F_d;\ZZ)=\colim_{n-1} H$ for $n\geq 2$, where  ${H\colon \CP_d\to \Ab}$ takes the following value on $U\in \CP_d$,
\[
H(U)=\{\sum_{g\in \Gamma_d} n_g\cdot g\in \ZZ[\Gamma_d]\mid \sum_{u\in U} n_{ug}=0\text{ for all $g\in \Gamma_d$}\},
\]
and takes morphisms to inclusions. By Whitehead's Theorem, $F_d$ is contractible if and only if the functor $H$ is acyclic. To check that this is indeed the case, we use below the shellability setup of Sections \ref{sec:vanishing} and \ref{sec:non-vanishing_stable_functor}. As a consequence, one could employ the Bousfield-Kan spectral sequence \cite[XII.5.7]{BK1972} to determine the (co)homology of the classifying spaces $B(\Gamma_d)$, see for instance \cite[Example 4.2]{CD23}. 

\subsubsection{Bianchi group $\Gamma_1$.} Consider the following presentation of $\Gamma_1$, see \cite{MR0372059},
\[
\Gamma_1=\langle a,b,c,d\mid a^3=b^2=c^3=d^2=(ac)^2=(ad)^2=(bc)^2=(bd)^2=1\rangle,
\]
as well as the following poset $\CP_1$ of proper subgroups of $\Gamma_1$,
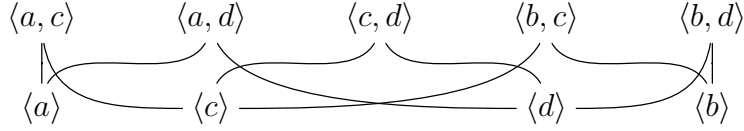
\[
\xymatrix@C=30pt@R=15pt{
{\langle a,c\rangle}& {\langle a,d\rangle}& {\langle c, d\rangle}& {\langle b,c\rangle}& {\langle b,d\rangle}\\
{\langle a\rangle}\ar@{-}@(u,d)[u]\ar@{-}@(u,d)[ru]& {\langle c\rangle}\ar@{-}@(l,d)[lu]\ar@{-}@(u,d)[ru]\ar@{-}@(r,d)[rru]&& {\langle d\rangle}\ar@{-}@(u,d)[lu]\ar@{-}@(r,d)[ru]\ar@{-}@(l,d)[llu]& {\langle b\rangle}\ar@{-}@(u,d)[lu]\ar@{-}@(u,d)[u]\\
}
\]
and the bounded poset $\CP=\CP_1\cup\{\hat 0, \hat 1\}$ with $\d(\CP)=3$. Moreover, by \cite[Section 7]{FLOGE1983} or  \cite[Theorem 4.4.1]{MR1010229}, the subgroup generated by $c$ and $d$ verifies:
\begin{equation}\label{equ:PSL2_Gamma1}
\langle c,d\rangle =\langle c\rangle \ast \langle d \rangle \simeq \PSL_2(\ZZ).
\end{equation}
We equip $\CP$ with the recursive coatom ordering given by
\[
\langle a,c\rangle\ll \langle a,d\rangle \ll \langle b,c\rangle\ll \langle b,d\rangle\ll  \langle c, d\rangle
\]
and via an easy computation determine that 
\[
C(\langle a,c\rangle)=\emptyset\text{, }C(\langle a,d\rangle)=\{\langle a\rangle\}\text{, }C(\langle b,c\rangle)=\{\langle c\rangle\}\text{, }C(\langle b,d\rangle)=\{\langle b\rangle,\langle d\rangle\}\text{, }C(\langle c, d\rangle)=\{\langle c\rangle,\langle d\rangle\}.
\]
Note that the functor $H$ is acyclic if and only if $H_1(\CP\setminus\{\hat 1\};H)=0$. The hypotheses of   Theorem \ref{thm:exact_sequence_for_atoms_covariant_colimit}  with $\CA=\{\langle b\rangle\}$ hold because of \cite[Theorem A, 4.1]{CD23} and \eqref{equ:PSL2_Gamma1}, and so we have a short exact sequence
\[
0\to H_1(\overline{\CP};H)\to H(\langle b\rangle)\stackrel{\delta}\to H_0(\overline{\CP}\setminus \{\langle b\rangle\};H)\to H_0(\overline{\CP};H)\to 0,
\]
where $\delta$ is the connecting homomorphism. Note that $H_1(\overline{\CP};H)=0$ if and only if $\delta$ is injective, and we show that this indeed is the case: An easy computation shows that 
\[
\delta(x_b)=[0\oplus 0\oplus 0\oplus \oplus x_b \oplus -x_b] 
\]
where $x_b\in H(\langle b\rangle)$ and $\delta(x_b)$ belongs to the quotient of 
\[
H(\langle a,c\rangle)\oplus H(\langle a,d\rangle)\oplus H(\langle c,d\rangle)\oplus H(\langle b,c\rangle)\oplus H(\langle b,d\rangle)
\]
by the equivalence relation determined by the colimit $ H_0(\overline{\CP}\setminus \{\langle b\rangle\};H)$. Thus, if $\delta(x_b)$ is zero, there exists $x_a,x'_a,x_c,x'_c,x''_c,x_d,x'_d,x''_d$ belonging to $H(\langle a\rangle),H(\langle c\rangle),H(\langle d\rangle)$ respectively with
\begin{align*}
x_a+x'_a&=0,x_c+x'_c+x''_c=0,x_d+x'_d+x''_d=0,\\
x_a+x_c&=0,x'_a+x_d=0,x'_c+x'_d=0,x''_c+x_b=0,x''_d-x_b=0.    
\end{align*}
Then we have $x_c+c_d=x'_c+x'_d=0$ and, by \cite[Theorems A, 4.1]{CD23} and \eqref{equ:PSL2_Gamma1}, also that $x_c=x_d=x'_c=x'_d=0$. Thus  $x''_c=x''_d=x_b=0$ and we are done.

\subsubsection{Bianchi group $\Gamma_2$.} Consider the following presentation of $\Gamma_2$, see \cite[Satz 6]{FLOGE1983},
\[
\Gamma_2=\langle a,m,s,v,u\mid a^2=s^3=(am)^2=1,am=sv^2, u^{-1}au=m,u^{-1}su=v\rangle,
\]
the following poset $\CP_2$ of proper subgroups of $\Gamma_2$,
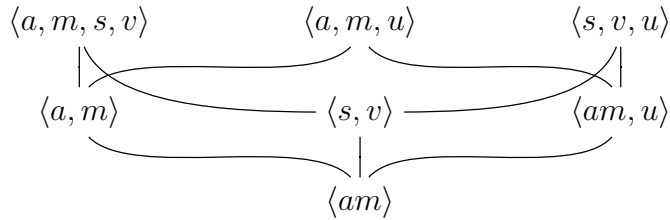
\[
\xymatrix@C=50pt@R=15pt{
{\langle a,m,s,v\rangle}& {\langle a,m,u\rangle}& {\langle s,v,u\rangle}\\
{\langle a,m\rangle}\ar@{-}@(u,d)[u]\ar@{-}@(u,d)[ru]& 
{\langle s,v\rangle}\ar@{-}@(l,d)[ul]\ar@{-}@(r,d)[ur]&
{\langle am,u\rangle}\ar@{-}@(u,d)[u]\ar@{-}@(u,d)[lu]\\
&{\langle am\rangle}\ar@{-}@(u,d)[ul]\ar@{-}@(u,d)[u]\ar@{-}@(u,d)[ur]
}
\]
and the bounded poset $\CP=\CP_2\cup\{ \hat 1\}$ with $\d(\CP)=3$. Note that
\begin{equation}\label{equ:amalgam_Gamma2}
\langle a,m,s,v\rangle = \langle a,m\rangle *_{\langle am\rangle}  \langle s,v\rangle\simeq K*_{C_2} A_4, 
\end{equation}
where $K$ is the Klein group and $A_4$ is the alternating group on 4 letters, see \cite[Satz 6]{FLOGE1983}. We equip $\CP$ with the recursive coatom ordering given by
\[
\langle s,v,u\rangle\ll \langle a,m,u\rangle \ll \langle a,m,s,v\rangle
\]
and we determine that
\[
C(\langle s,v,u\rangle)=\emptyset\text{, }C(\langle a,m,u\rangle)=\{\langle am,u\rangle\}\text{, and }C(\langle a,m,s,v\rangle)=\{\langle a,m\rangle,\langle s,v\rangle\}.
\]
Because $H$ is a monic functor and because of \cite[Theorems A, 4.1]{CD23} and \eqref{equ:amalgam_Gamma2}, it is easy to check that $H$ has the stability property at $i=1,2$, so that, by \cref{dual/main_theorem}, $H$ is acyclic.

\bibliographystyle{alpha}
\bibliography{biblio.bib}
\end{document}